\documentclass[12pt]{amsart}

\newtheorem{theorem}{Theorem}[section]
\newtheorem*{mainthm}{Theorem}

\newtheorem{lemma}[theorem]{Lemma}

\newtheorem{proposition}[theorem]{Proposition}
\newtheorem{corollary}[theorem]{Corollary}

\theoremstyle{remark}

\newtheorem*{remark}{Remark}

\newtheorem*{notation}{Notation}
\newtheorem*{acknowledgment}{Acknowlegment}

\newcommand{\et}{\quad\mbox{and}\quad}

\newcommand{\bQ}{\mathbb{Q}}
\newcommand{\bR}{\mathbb{R}}
\newcommand{\bZ}{\mathbb{Z}}

\newcommand{\C}[1]{C_{#1}}
\newcommand{\Cmoins}{C^-}
\newcommand{\Cplus}{C^+}

\newcommand{\petit}[1]{L(#1)}
\newcommand{\ssi}{\ \Longleftrightarrow\ }

\newcommand{\uw}{\mathbf{w}}
\newcommand{\ux}{\mathbf{x}}

\newcommand{\uxmoins}{\ux^-}
\newcommand{\uxplus}{\ux^+}
\newcommand{\uy}{\mathbf{y}}
\newcommand{\uymoins}{\uy^-}
\newcommand{\uyplus}{\uy^+}
\newcommand{\uz}{\mathbf{z}}

\begin{document}

\baselineskip=16.5pt%20pt

\title[Simultaneous rational approximations]
{On simultaneous rational approximations to a real number, its
square, and its cube}
\author{Damien ROY}
\address{
   D\'epartement de Math\'ematiques\\
   Universit\'e d'Ottawa\\
   585 King Edward\\
   Ottawa, Ontario K1N 6N5, Canada}
\email{droy@uottawa.ca}
\thanks{Work partially supported by NSERC and CICMA}

\begin{abstract}
We show that, for any transcendental real number $\xi$, the uniform
exponent of simultaneous approximation of the triple $(\xi, \xi^2,
\xi^3)$ by rational numbers with the same denominator is at most $(1
+ 2\gamma - \sqrt{1+4\gamma^2})/2 \cong 0.4245$ where
$\gamma=(1+\sqrt{5})/2$ stands for the golden ratio.  As a
consequence, we get a lower bound on the exponent of approximation
of such a number $\xi$ by algebraic integers of degree at most $4$.
\end{abstract}

\maketitle

\section{Introduction}
 \label{sec:intro}

In a remarkable paper \cite{DSb}, H.~Davenport and W.~M.~Schmidt
showed that, for any integer $n\ge 2$ and for any real number $\xi$
which is not algebraic over $\bQ$ of degree at most $n-1$, there
exist infinitely many algebraic integers $\alpha$ of degree at most
$n$ satisfying
\[
 |\xi-\alpha| \le c H(\alpha)^{-\tau(n)}
\]
where $c=c(n,\xi)>0$ is an appropriate constant depending only on
$n$ and $\xi$, and where $\tau(2)=2$, $\tau(3)=(3+\sqrt{5})/2$,
$\tau(4)=3$ and $\tau(n)=\lfloor (n+1)/2 \rfloor$ if $n\ge 5$. For
$n=2,3$, this value of $\tau(n)$ cannot be improved (see \cite{DSb}
for the case $n=2$ and \cite{Rc} for the case $n=3$). For $n\ge 4$,
M.~Laurent showed in \cite{La} that $\tau(n)$ can be taken to be
$\lceil (n+1)/2 \rceil$. However, at present, no optimal value for
$\tau(n)$ is known for any single value of $n\ge 4$. Furthermore, we
possess no non-trivial upper bound for $\tau(n)$ for $n\ge 4$,
besides the estimate $\tau(n)\le n$ coming from metrical
considerations (by an application of the Borel-Cantelli lemma as in
the proof of \cite[Thm.~3.3]{Bu}). Although, we shall not go into
this, let us simply mention that the situation is similar in the
case of approximation by algebraic numbers of degree at most $n$. In
this case, it is only for $n\le 2$ that the optimal exponents are
known, the case $n=2$ being due once again to Davenport and Schmidt
\cite{DSa}.

Several years ago, I started working on finding an optimal value for
$\tau(4)$ (in the above notation) and, despite of much effort, I was
not successful.  My hopes were that this would lead to a new class
of extremal numbers, similar to that of \cite{Ra} or \cite[\S
6]{Rb}, and that such construction could be generalized to larger
values of $n$ to provide a non-trivial upper-bound for the
corresponding values of $\tau(n)$, and maybe settle the question as
to whether $\limsup_{n\to\infty} \tau(n)/n$ is equal to $1$ or
strictly smaller than $1$.  These problems remain open.

The method initiated by Davenport and Schmidt in \cite{DSb} for
estimating $\tau(n)$ is based on geometry of numbers and requires an
upper bound on the uniform exponent of simultaneous approximation of
the first $n-1$ consecutive powers of a real number $\xi$ by
rational numbers with the same denominator.  By \cite[\S 2, Lemma
1]{DSb}, our main result below implies that $\tau(4)$ can be taken
to be $\lambda_3^{-1}+1 \cong 3.3556$, where
\[
 \lambda_3
   = \frac{1}{2} \left(2+\sqrt{5}-\sqrt{7+2\sqrt{5}}\right)
   \cong 0.4245.
\]

\begin{mainthm}
Let $\xi\in\bR$ with $[\bQ(\xi)\colon\bQ] >3$, and let $c$ and
$\lambda$ be positive real numbers.  Suppose that for any
sufficiently large value of $X$, the inequalities
\begin{equation}
 \label{main:eq}
 |x_0|\le X, \quad
 |x_0\xi-x_1| \le cX^{-\lambda}, \quad
 |x_0\xi^2-x_2| \le cX^{-\lambda}, \quad
 |x_0\xi^3-x_3| \le cX^{-\lambda},
\end{equation}
admit a non-zero solution $\ux=(x_0,x_1,x_2,x_3)\in \bZ^4$. Then, we
have $\lambda\le \lambda_3$.  Moreover, if $\lambda=\lambda_3$, then
$c$ is bounded below by a positive constant depending only on $\xi$.
\end{mainthm}

The rest of the paper is devoted to the proof of this result which,
through its weaker hypothesis on $\xi$, complements \cite[Theorem
4a]{DSb}. The tools that we use for the proof are the same as those
of \cite{DSb} together with results on heights of subspaces of
$\bR^n$ defined over $\bQ$ that were developed around the same
period of time by W.~M.~Schmidt in \cite{Sa}.  Using other tools,
similar to the bracket $[\ux,\uy,\uz]$ in \cite[\S2]{Rb}, I
discovered recently that the exponent $\lambda_3$ in the above
theorem is not optimal. Since the argument is quite involved and
does not seem to lead to a significant improvement in $\lambda_3$, I
decided not to include this here.

%%%%%%%%%%%%%%%%%%%%%%%%%%%%%%%%%%%%%%%%%%%%%%%%%%%%%%%%%%%%%%%%%
%
%   First considerations
%
%%%%%%%%%%%%%%%%%%%%%%%%%%%%%%%%%%%%%%%%%%%%%%%%%%%%%%%%%%%%%%%%%

\section{First considerations}
\label{sec:first}

Throughout this paper, we fix a real number $\xi$ with
$[\bQ(\xi)\colon\bQ] >3$ and positive constants $\lambda$, $c$\,
satisfying the hypotheses of the Theorem.   In all statements below,
the implied constants in the symbols $\gg$, $\ll$ and $\asymp$ (the
conjunction of $\gg$ and $\ll$) depend only on $\xi$ and $\lambda$
(not on $c$). In particular, we may assume that $c \ll 1$.  Our goal
is to show that $\lambda\le \lambda_3$ and that $c\gg 1$ in case of
equality.  By \cite[Theorem 4a]{DSb}, we already have $\lambda \le
1/2$.

For each integer $n\ge 1$ and each point $\ux = (x_0, x_1, \dots,
x_n) \in \bR^{n+1}$, we define points of $\uxmoins$ and $\uxplus$ of
$\bR^n$ by
\[
 \uxmoins=(x_0,\dots,x_{n-1})
  \et
 \uxplus=(x_1,\dots,x_n).
\]
We also put
\[
 \|\ux\|=\max_{0\le i \le n} |x_i|
  \et
 L(\ux)=\max_{1\le i\le n} |x_0\xi^i-x_i|.
\]
Finally, we say that a point $\ux\in\bZ^{n+1}$ is \emph{primitive}
if it is non-zero and if the gcd of its coordinates is $1$.  Then,
the hypothesis implies that, for any sufficiently large $X$, there
exist a primitive point $\ux\in\bZ^4$ with
\begin{equation}
 \label{main_cond}
 \|\ux\| \le X \et L(\ux) \le c c_1 X^{-\lambda},
\end{equation}
where $c_1=2\max\{1,|\xi|\}^{3\lambda}$.  The following lemmas
extend results of Davenport and Schmidt in \cite[\S 4]{DSb}.

\begin{lemma}
 \label{lemma:pointy}
Let $C\in\bZ^2$ and\ $\ux \in \bZ^{n+1}$ with $n\in\{1,2,3\}$. Then,
$\uy = \Cplus\uxmoins - \Cmoins\uxplus$ satisfies
\begin{equation}
 \label{lemma:pointy:eq}
 \|\uy\| \le \|\ux\| L(C) + c_2 \|C\| L(\ux)
 \et
 L(\uy) \le c_2 \|C\| L(\ux)
\end{equation}
for some constant $c_2=c_2(\xi)$.  Moreover if $\uy=0$ and if $C$
and $\ux$ are non-zero and primitive, we have
 \[
 \|\ux\|= \|C\|^n \et L(\ux)  \asymp \|C\|^{n-1}L(C).
 \]
\end{lemma}

\begin{proof}
Write $C=(a,b)$.  Then, the estimates in \eqref{lemma:pointy:eq}
follow respectively from the formulas $\uy = (b-a\xi)\uxmoins + a
(\xi\uxmoins-\uxplus)$ and $\uy=b\ux^+-a\ux^-$, upon choosing $c_2$
so that $\|\xi\uxmoins-\uxplus\| \le c_2 L(\ux)$ and
$L(\uxmoins)+L(\uxplus) \le c_2 L(\ux)$. If $\uy=0$ and $C\neq 0$,
then $\ux$ is a rational multiple of the geometric progression
$(a^{n},a^{n-1}b,\dots,b^{n})$. If furthermore $C$ and $\ux$ are
primitive, this progression is a primitive point of $\bZ^{n+1}$ and
so it coincides with $\pm \ux$. This gives $\|\ux\| = \|C\|^{n}$ and
$L(\ux) \asymp \|\uxplus-\xi\uxmoins\| = \|C\|^{n-1}L(C)$.
\end{proof}

\begin{lemma}
 \label{lemma:minoreLC}
Suppose that $\lambda>1/3$.  Then for any non-zero point $C\in\bZ^2$
we have $L(C)\gg \|C\|^{-1/\lambda}$.
\end{lemma}

\begin{proof}
Since $\xi\notin \bQ$, we have $L(C)\neq 0$ for any non-zero point
$C\in\bZ^2$.  So, it suffices to prove that $L(C)\gg
\|C\|^{-1/\lambda}$ for primitive points $C\in\bZ^2$ of sufficiently
large norm. Let $C$ be a primitive point of $\bZ^2$, and let
$\ux\in\bZ^4$ be a primitive solution of \eqref{main_cond} for the
choice of $X=(2cc_1c_2\|C\|)^{1/\lambda}$, where $c_2$ is the
constant introduced in Lemma \ref{lemma:pointy}.  Since
$\lambda>1/3$, we have $X < \|C\|^3$ if $\|C\| \gg 1$, and then the
second part of Lemma \ref{lemma:pointy} shows that $\uy =
\Cplus\uxmoins - \Cmoins\uxplus$ is a non-zero point of $\bZ^3$.
Applying the first part of the same lemma, we deduce that
\[
 1 \le \|\uy\|
   \ll X L(C) + c c_1 c_2 \|C\| X^{-\lambda}
   \ll X L(C) + 1/2,
\]
and so $L(C) \ge (2X)^{-1} \gg \|C\|^{-1/\lambda}$.
\end{proof}

\begin{lemma}
 \label{lemma:xmoinsxplus}
Suppose that $\lambda >1/3$.  Then, there exist at most finitely
many points $\ux\in\bZ^4$ with $L(\ux) \le c c_1 \|\ux\|^{-\lambda}$
such that $\uxmoins$ and $\uxplus$ are linearly dependant over
$\bQ$.
\end{lemma}

\begin{proof}
Suppose on the contrary that the conclusion is false. Then, there
exist infinitely many primitive points $\ux$ of $\bZ^4$ with
$L(\ux)\le c c_1 \|\ux\|^{-\lambda}$ for which $\uxmoins$ and
$\uxplus$ are linearly dependant.  For each of them, there exists a
primitive point $C\in\bZ^2$ such that $\Cplus\uxmoins -
\Cmoins\uxplus=0$. By Lemma \ref{lemma:pointy}, we have $\|\ux\|=
\|C\|^3$ and $L(\ux) \asymp \|C\|^2 L(C)$.  Thus $\|C\|$ tends to
infinity with $\|\ux\|$, and the condition $L(\ux) \le c c_1
\|\ux\|^{-\lambda}$ translates into $L(C) \ll c
\|C\|^{-2-3\lambda}$. Since $-2-3\lambda < -3 < -1/\lambda$, this
contradicts Lemma \ref{lemma:minoreLC}.
\end{proof}

\begin{lemma}
 \label{lemma:fonctionL}
Let $n\in\{1,2,3\}$ and let $U$ be a proper subspace of $\bR^{n+1}$
defined over $\bQ$. Then, the function $L(\ux)$ is bounded from
below by a positive constant on the set of all non-zero points $\ux$
of $U\cap \bZ^{n+1}$.
\end{lemma}

\begin{proof}
As in the proof of \cite[\S3, Lemma 5]{DSb}, suppose on the contrary
that there exists a sequence of non-zero integral points
$(\ux_i)_{i\ge 1}$ in $U$ such that $\lim_{i\to\infty}L(\ux_i)=0$.
Then, for any sufficiently large index $i$, the first coordinate
$x_{i,0}$ of $\ux$ is non-zero and the product $x_{i,0}^{-1}\ux_i$
converges to $(1,\xi,\dots,\xi^n)$ as $i$ tends to infinity.  Thus,
the point $(1,\xi,\dots,\xi^n)$ belongs to $U$.  This is impossible
since $U$ is a proper subspace of $\bR^{n+1}$ defined over $\bQ$
while the coordinates of the point $(1,\xi,\dots,\xi^n)$ are
linearly independent over $\bQ$.
\end{proof}

Finally, we note that there exists a sequence of non-zero points
$(\ux_i)_{i\ge 1}$ in $\bZ^4$ with the following properties:
\begin{itemize}
 \item[(a)] the positive integers $X_i:=\|\ux_i\|$ form a strictly
 increasing sequence,
 \item[(b)] the positive real numbers $L_i:=L(\ux_i)$ form a
 strictly decreasing sequence,
 \item[(c)] if some non-zero point $\ux\in\bZ^4$ satisfies $L(\ux)
 < L_i$ for some $i\ge 1$, then $\|\ux\| \ge
 X_{i+1}$.
\end{itemize}
We fix such a choice of sequence $(\ux_i)_{i\ge 1}$ and refer to it
as the sequence of \emph{minimal points} for $\xi$ although it is
not unique and differs from the notion introduced by Davenport and
Schmidt in \cite[\S4]{DSb}.  We note that, for each $i\ge 1$,
$\ux_i$ is a primitive point of $\bZ^4$ and, since \eqref{main_cond}
admits a non-zero solution $\ux\in\bZ^4$ for each $X$ with $X_i \le
X < X_{i+1}$ when $i$ is sufficiently large, we deduce from the
condition (c) that
\[
 L_i \le c c_1 X_{i+1}^{-\lambda}
\]
for each large enough index $i$.  We will use this property
repeatedly in the sequel, either in this form or in the weaker form
$L_i \ll c X_{i+1}^{-\lambda} \ll X_{i+1}^{-\lambda}$.

%%%%%%%%%%%%%%%%%%%%%%%%%%%%%%%%%%%%%%%%%%%%%%%%%%%%%%%%%%%%%%%%%
%
%   A family of planes in $\bR^4$
%
%%%%%%%%%%%%%%%%%%%%%%%%%%%%%%%%%%%%%%%%%%%%%%%%%%%%%%%%%%%%%%%%%

\section{A family of planes in $\bR^4$}
\label{sec:W}

For each integer $n\ge 1$ and each subspace $S$ of $\bR^n$ defined
over $\bQ$ of dimension $p>0$, we define the \emph{height} $H(S)$ of
$S$ by $H(S) = \| \uy_1 \wedge \cdots \wedge \uy_p \|$ where
$(\uy_1, \dots, \uy_p)$ is a basis of the group $S\cap\bZ^n$ of
integral points of $S$ (upon identifying $\bigwedge^p \bR^n$ with
$\bR^{\binom{n}{p}}$ through an ordering of the Grassmann
coordinates, as in \cite[Chap.~1, \S5]{Sb}).  We also define
$H(0)=1$.  Then, it follows from \cite[Chap.~1, Lemma 8A]{Sb} that,
for any pair of subspaces $S$ and $T$ of $\bR^n$ defined over $\bQ$,
we have
\begin{equation}
 \label{est:height_sum_inter}
 H(S\cap T)H(S+T) \le c(n) H(S) H(T)
\end{equation}
with a constant $c(n)>0$ depending only on $n$.  We also recall that
$H(S)=H(S^\perp)$ where $S^\perp$ stands for the orthogonal
complement of $S$ in $\bR^n$ (see \cite[Chap.~1, \S8]{Sb}).

For each $i\ge 2$, we denote by $W_i$ the subspace of $\bR^4$ of
dimension $2$ generated by $\ux_{i-1}$ and $\ux_i$. We also
introduce a new parameter
\[
 \theta = \frac{1-\lambda}{\lambda},
\]
and note that $\theta\ge 1$ since $\lambda\le 1/2$.

\begin{lemma}
 \label{lemma:hauteurWi}
For each $i\ge 2$, the points $\ux_{i-1}$ and $\ux_i$ form a basis
of $W_i\cap\bZ^4$, and we have: \quad $H(W_i) \asymp X_i L_{i-1} \ll
X_i^{1-\lambda}$.
\end{lemma}

This follows by a simple adaptation of the proofs of \cite[Lemma
2]{DSa} and \cite[Lemma 4.1]{Rb}, the difference being that here
$X_i$ stands for the norm of $\ux_i$ instead of the absolute value
of its first coordinate.  We now look at sums $W_i+W_{i+1}$.

\begin{lemma}
 \label{lemma:Hsomme}
There exist infinitely many indices $i\ge 2$ such that $W_i\neq
W_{i+1}$. For each of them, we have
\begin{equation}
 \label{est:HWi+Wi+1}
 H(W_i+W_{i+1})
   \ll X_i^{-1} H(W_i) H(W_{i+1})
   \ll H(W_i)^{-1/\theta} H(W_{i+1}).
\end{equation}
\end{lemma}

\begin{proof}
If there were only finitely many indices $i\ge 2$ for which $W_i\neq
W_{i+1}$, then all points $\ux_i$ with $i$ sufficiently large would
lie in a fixed subspace $W$ of $\bR^4$ defined over $\bQ$, against
Lemma \ref{lemma:fonctionL}. This proves the first assertion of the
lemma.

Applying \eqref{est:height_sum_inter} with $S=W_i$ and $T=W_j$, we
find
\[
 H(W_i\cap W_{i+1}) H(W_i+W_{i+1}) \ll H(W_i) H(W_{i+1}).
\]
For each index $i\ge 2$ such that $W_i\neq W_{i+1}$, we have
$W_i\cap W_{i+1} = \langle\ux_i\rangle_\bR$ and so $H(W_i\cap
W_{i+1})=X_i$. This leads to the first estimate in
\eqref{est:HWi+Wi+1}. For the second one, we simply use the upper
bound $X_i\gg H(W_i)^{1/(1-\lambda)}$ coming from Lemma
\ref{lemma:hauteurWi}.
\end{proof}

\begin{notation}
We denote by $I$ the set of indices $i\ge 2$ for which $W_i\neq
W_{i+1}$.
\end{notation}

Thus, for each $i\in I$, the sum $W_i + W_{i+1}=\langle
\ux_{i-1},\ux_i,\ux_{i+1}\rangle_\bR$ is a 3-dimensional subspace of
$\bR^4$ defined over $\bQ$.  By Lemma \ref{lemma:fonctionL} such a
subspace of $\bR^4$ contains at most finitely many minimal points.
This leads to the first assertion of the next lemma.

\begin{lemma}
 \label{lemma:paire}
There exist infinitely many pairs of consecutive elements $i,\,j$ of
$I$ with $i<j$ and $W_i+W_{i+1} \neq W_j+W_{i+1}$. For such a pair
of integers, we have
\begin{gather}
 \label{est:generale}
 X_{i} X_{j} \ll H(W_i) H(W_j) H(W_{j+1}),\\
 \label{est:paire}
 H(W_i) H(W_j) \ll H(W_{j+1})^\theta
 \et
 X_{i} X_{j} \ll X_{j+1}^\theta.
\end{gather}
\end{lemma}

\begin{proof}
For consecutive elements $i<j$ of $I$, we have $W_i\neq W_{i+1} =
W_j \neq W_{j+1}$.  If $W_i+W_{i+1}$ and $W_j+W_{j+1}$ are distinct
subspaces of $\bR^4$, their sum is the whole of $\bR^4$ and their
intersection is $W_{i+1}=W_j$.  Since $H(\bR^4)=1$, we deduce from
\eqref{est:height_sum_inter} that
$$
H(W_{i+1}) \ll H(W_i+W_{i+1}) H(W_j+W_{j+1})
$$
Combining this estimate with the upper bounds
$$
 H(W_i+W_{i+1}) \ll X_{i}^{-1} H(W_i) H(W_{i+1})
 \et
 H(W_j+W_{j+1}) \ll X_{j}^{-1} H(W_j) H(W_{j+1})
$$
provided by Lemma \ref{lemma:Hsomme}, we obtain
\eqref{est:generale}. Then combining \eqref{est:generale} with the
standard upper bounds $H(W_i)\ll X_{i}^{1-\lambda}$ and $H(W_j) \ll
X_{j}^{1-\lambda}$ coming from Lemma \ref{lemma:hauteurWi}, we find
$$
 X_{i}^\lambda X_{j}^\lambda \ll H(W_{j+1}),
$$
and so $
 H(W_i) H(W_j)
  \ll (X_{i} X_{j})^{1-\lambda}
  \ll H(W_{j+1})^\theta
  \ll X_{j+1}^{\theta(1-\lambda)}
$, which proves \eqref{est:paire}.
\end{proof}

%%%%%%%%%%%%%%%%%%%%%%%%%%%%%%%%%%%%%%%%%%%%%%%%%%%%%%%%%%%%%%%%%
%
%   A family of points in $\bZ^2$
%
%%%%%%%%%%%%%%%%%%%%%%%%%%%%%%%%%%%%%%%%%%%%%%%%%%%%%%%%%%%%%%%%%

\section{A family of points in $\bZ^2$}
\label{sec:C}

For each pair of points $\ux$ and $\uy$ in $\bZ^4$, we define
\[
 C(\ux,\uy)
 = (\det(\uxmoins,\uxplus,\uymoins),
    \det(\uxmoins,\uxplus,\uyplus)) \in \bZ^2.
\]
To alleviate the notation, we also write
\[
 C_{i,j}= C(\ux_i,\ux_j)
\]
for each pair of integers $i,j\ge 1$.  These points $C_{i,j}$ play a
crucial role in the proof of the inequality $\lambda\le 1/2$ by
Davenport and Schmidt in \cite[\S4]{DSb}.  They also play an
important role in the present work.  We first prove general
estimates.

\begin{lemma}
 \label{lemma:Cij}
For any pair of integers $i,j\ge 1$, we have
\[
 \|C_{i,j}\|
  \ll X_jL_i^2+X_iL_iL_j
 \et
 L(C_{i,j}) \ll X_iL_iL_j.
\]
\end{lemma}

\begin{proof}
The estimate for $\|C_{i,j}\|$ is standard (see for example the
proof of \cite[\S4, Lemma 7]{DSb}).  For the other quantity, we find
\[
 L(C_{i,j})
  = |\det(\uxmoins_i,\uxplus_i,\uxplus_j-\xi\uxmoins_j)|
  = |\det(\uxmoins_i,\uxplus_i-\xi\uxmoins_i,\uxplus_j-\xi\uxmoins_j)|
  \ll X_iL_iL_j.
\]
\end{proof}

The next lemma provides a sharper upper bound for $L(\C{i,i+1})$
when $i\in I$.

\begin{lemma}
 \label{lemma:L(Ci,i+1)}
Let $i<j$ be consecutive elements of $I$. Then, we have
$C_{i,j}=b\C{i,i+1}$ for some non-zero integer $b$ with $|b| \asymp
X_j/X_{i+1}$, and
\[
L(\C{i,i+1}) \ll X_i X_j^{-\lambda} X_{j+1}^{-\lambda}.
\]
\end{lemma}

\begin{proof}
Since $i$ and $j$ are consecutive in $I$, we have $W_{i+1}=W_j$.
Moreover since $\ux_i$ and $\ux_{i+1}$ form a basis of the group of
integral points of $W_{i+1}$, there exist non-zero integers $a$ and
$b$ such that $\ux_j = a\ux_i+b\ux_{i+1}$. If $X_j > 3|b|X_{i+1}$,
we deduce that
\[
 |a| X_i
    = \| \ux_j - b\ux_{i+1} \|
   \ge X_j- |b| X_{i+1}
    > 2 |b| X_{i+1},
\]
and so $|a| > 2 |b|$. Then, we find
 $
 L_j \ge |a| L_i - |b| L_{i+1} > |b| L_{i+1} \ge L_{i+1}
 $,
which is impossible.  This contradiction shows that $|b| \ge X_j /
(3X_{i+1})$. Since the point $C(\ux,\uy)$ is a linear function of
$\uy$ and since $C(\ux,\ux)=0$ for any $\ux\in\bR^4$, we also have
\[
 C_{i,j} = C(\ux_i,a\ux_i+b\ux_{i+1}) = b \C{i,i+1}
\]
and so, by Lemma \ref{lemma:Cij}, we obtain (since $\lambda \le 1/2
\le 1$)
\[
 L(\C{i,i+1})
  = |b|^{-1} L(C_{i,j})
  \le |b|^{-\lambda} L(C_{i,j})
  \ll \frac{X_{i+1}^\lambda}{X_j^\lambda} X_i L_i L_j
  \ll X_i X_j^{-\lambda} X_{j+1}^{-\lambda}.
\]
\end{proof}

\begin{remark}
Although we will not use this here, it is interesting to note that
the identity
\[
   \det(\uw,\ux,\uy)\uz
 - \det(\uw,\ux,\uz)\uy
 + \det(\uw,\uy,\uz)\ux
 - \det(\ux,\uy,\uz)\uw =0,
\]
which holds for any quadruple of points $(\uw,\ux,\uy,\uz)$ in
$\bR^3$, specializes to
\[
 \Cplus_{i,j}\uxmoins_j - \Cmoins_{i,j}\uxplus_j
   = \Cmoins_{j,i}\uxplus_i - \Cplus_{j,i}\uxmoins_i.
\]
when we apply it to the quadruple $(\uxmoins_i, \uxplus_i,
\uxmoins_j, \uxplus_j)$ for a choice of integers $i,j\ge 1$.
\end{remark}

%%%%%%%%%%%%%%%%%%%%%%%%%%%%%%%%%%%%%%%%%%%%%%%%%%%%%%%%%%%%%%%%%
%
%   A family of planes in $\bR^3$
%
%%%%%%%%%%%%%%%%%%%%%%%%%%%%%%%%%%%%%%%%%%%%%%%%%%%%%%%%%%%%%%%%%

\section{A family of planes in $\bR^3$}
\label{sec:V}

    From now on, we assume that $\lambda>1/3$.  Then, by Lemma
\ref{lemma:xmoinsxplus}, there exists an index $i_0$ such that
$\uxmoins_i$ and $\uxplus_i$ are linearly independent for each $i\ge
i_0$. For those values of $i$, we denote by $V_i$ the
two-dimensional subspace of $\bR^3$ spanned by these points:
\[
 V_i = \langle \uxmoins_i,\uxplus_i \rangle_\bR.
\]
Since $\max\{L(\uxmoins_j), L(\uxplus_j)\} \ll L_j$ tends to $0$ as
$j\to \infty$, it follows from Lemma \ref{lemma:fonctionL} that each
$V_i$ contains at most finitely many points of the form $\uxmoins_j$
or $\uxplus_j$, and so there are infinitely many indices $i\ge i_0$
such that $V_i\neq V_{i+1}$.  We also note that, for $i,j\ge i_0$,
we have
\[
 V_i=V_j \ssi C_{i,j}=0 \ssi C_{j,i}=0
\]
by definition of the points $C_{i,j}$ (see \S\ref{sec:C}). In
\cite[\S4]{DSb}, Davenport and Schmidt argue that, for each $i\ge
i_0$ such that $V_i\neq V_{i+1}$, we have $1\le \|C_{i,i+1}\| \ll
X_{i+1}L_i^2 \ll X_{i+1}^{1-2\lambda}$ (see Lemma \ref{lemma:Cij}).
Since $i$ can be taken to be arbitrarily large, this gives
$1-2\lambda\ge 0$ and so $\lambda\le 1/2$.

\begin{lemma}
 \label{lemma:normeDi}
There exist infinitely many integers $i>i_0$ for which $V_{i-1}\neq
V_i$. For each of them, we have,
\begin{equation}
 \label{est:Xi+1versusXi}
 H(W_{i+1})
  \ll X_{i+1}^{1-\lambda}
  \ll H(W_i)^\theta
  \ll X_i^{\theta(1-\lambda)}.
\end{equation}
\end{lemma}

In particular, this leads to symmetric estimates $X_{i+1}\ll
X_i^\theta$ and $H(W_{i+1})\ll H(W_i)^\theta$.

\begin{proof}
The first assertion being already settled, fix an index $i>i_0$ such
that $V_{i-1}\neq V_i$.  Then the integral point $C_{i,i-1}$ is
non-zero and so its norm is bounded below by $1$.  The absolute
values of its coordinates are:
\begin{align*}
 |\det(\uxmoins_i,\uxplus_i,\uxmoins_{i-1})|
   &= |\det(\uxmoins_{i-1},\uxmoins_i,\uxplus_i - \xi\uxmoins_i)|
   \ll \|\uxmoins_{i-1}\wedge\uxmoins_i\| L_i, \\
 |\det(\uxmoins_i,\uxplus_i,\uxplus_{i-1})|
   &= |\det(\uxplus_{i-1},\uxplus_i,\uxmoins_i - \xi^{-1}\uxplus_i )|
   \ll \|\uxplus_{i-1}\wedge\uxplus_i\| L_i.
\end{align*}
Since $\|\uxmoins_{i-1}\wedge\uxmoins_i\|$ and
$\|\uxplus_{i-1}\wedge\uxplus_i\|$ are bounded above by
$\|\ux_{i-1}\wedge\ux_i\| = H(W_i)$, this means that $\|C_{i,i-1}\|
\ll H(W_i)L_i$.  Thus we obtain
\[
 1\le \|C_{i,i-1}\| \ll H(W_i) L_i \ll H(W_i) X_{i+1}^{-\lambda},
\]
and so $X_{i+1}\ll H(W_i)^{1/\lambda}$.  The conclusion follows by
combining this result with the estimates $H(W_i) \ll
X_i^{1-\lambda}$ and $H(W_{i+1})\ll X_{i+1}^{1-\lambda}$ coming from
Lemma \ref{lemma:hauteurWi}.
\end{proof}

\begin{proposition}
 \label{prop:Vi}
Suppose that there exist infinitely many indices $i\ge i_0$ such
that $V_i = V_{i+1}$. Then we have  $\lambda \le \sqrt{2}-1 \cong
0.4142$. Moreover, if $\lambda = \sqrt{2}-1$, then we also have $c
\gg 1$.
\end{proposition}

\begin{proof}
Since there are infinitely many indices $i>i_0$ for which
$V_{i-1}\neq V_i$, the hypothesis of the proposition forces the
existence of arbitrarily large indices $i$ with
$$
 V_{i-1}\neq V_i = V_{i+1}.
$$
Fix such an integer $i$.  Let $px_0+qx_1+rx_2=0$ be an equation of
$V_i$ with relatively prime coefficients $p,q,r\in \bZ$, so that by
duality $H(V_i)=\|(p,q,r)\|$.  For any point $\ux=(x_0,x_1,x_2,x_3)$
of $W_{i+1}$, we have
$$
 \ux^-=(x_0,x_1,x_2) \in \langle \ux_i^-,\ux_{i+1}^- \rangle_\bR
 \et
 \ux^+=(x_1,x_2,x_3) \in \langle \ux_i^+,\ux_{i+1}^+ \rangle_\bR,
$$
therefore $\ux^-$ and $\ux^+$ both belong to $V_i+V_{i+1} = V_i$,
and so the point $\ux$ satisfies
$$
 px_0+qx_1+rx_2=0 \et px_1+qx_2+rx_3=0.
$$
This means that the orthogonal complement of $W_i$ in $\bR^4$ is
$\langle (p,q,r,0),(0,p,q,r) \rangle_\bR$ and so, applying the
duality property of the height again, we find
\begin{equation}
 \label{prop:Vi:eq1}
 H(W_{i+1})
  = H(\langle (p,q,r,0),(0,p,q,r) \rangle_\bR)
  \asymp \|(p,q,r)\|^2
  = H(V_i)^2
\end{equation}
(the relation $H(V_i)\ll H(W_{i+1})^{1/2}$ also follows from
\cite[Thm.~3]{DSb} since the equality $V_i=V_{i+1}$ means that
$(p,q,r)$ provides a three terms recurrence relation satisfied both
by $\ux_i$ and $\ux_{i+1}$).  We now argue as M.~Laurent in the
proof of \cite[Lemma 5]{La}.  Define
$$
P(T) = p+qT+rT^2 \in \bZ[T].
$$
For any point $\uy=(y_0,y_1,y_2)\in\bZ^3$, we have
\begin{equation}
 \label{prop:Vi:eq2}
|(py_0+qy_1+ry_2)-y_0P(\xi)|
  \le 2 H(V_i) \petit{\uy}.
\end{equation}
Applying this estimate to the point $\uy = \ux_{i+1}^- \in V_i$, we
get
\begin{equation}
 \label{prop:Vi:eq3}
X_{i+1}|P(\xi)| \ll H(V_i) L_{i+1}.
\end{equation}
Since $V_{i-1}\neq V_i$, at least one of the points $\ux_{i-1}^-$ or
$\ux_{i-1}^+$ does not belong to $V_i$. If $\uy=(y_0,y_1,y_2)$ is
such a point, then $py_0+qy_1+ry_2$ is a non-zero integer, and using
successively \eqref{prop:Vi:eq2}, \eqref{prop:Vi:eq3} and
\eqref{prop:Vi:eq1} we obtain
\[
1 \le |py_0+qy_1+ry_2|
  \ll X_{i-1}|P(\xi)| + H(V_i) L_{i-1}
  \ll H(V_i) L_{i-1}
  \ll c H(W_{i+1})^{1/2} X_i^{-\lambda}.
\]
Moreover, Lemma \ref{lemma:normeDi} gives $H(W_{i+1}) \ll
X_i^{\theta(1-\lambda)}$ and so the last estimate leads to
\[
1 \ll c X_i^{(1-\lambda)^2/(2\lambda)-\lambda}
   = c X_i^{(2-(1+\lambda)^2)/(2\lambda)}.
\]
As $i$ can be taken to be arbitrarily large, this implies that
$2-(1+\lambda)^2\ge 0$, and so $\lambda\le \sqrt{2}-1$.  Moreover,
we obtain $c\gg 1$ if $\lambda=\sqrt{2}-1$.
\end{proof}

\begin{corollary}
 \label{cor:hauteurs}
Suppose that $\lambda>\sqrt{2}-1$.  Then, we have $V_{i-1}\neq V_i$
for any sufficiently large integer $i$, and the estimates
\eqref{est:Xi+1versusXi} of Lemma \ref{lemma:normeDi} apply to all
integers $i\ge 1$.  Moreover, for any pair of consecutive integers
$i<j$ of $I$ with $W_i+W_{i+1}\neq W_j+W_{j+1}$, we also have
\begin{align}
 \label{cor:hauteurs:estij}
 H(W_i)
  &\ll X_i^{1-\lambda}
  \ll\ \ H(W_j)^{\theta^2-1} \ \
  \ll X_j^{(\theta^2-1)(1-\lambda)} \\
 \label{cor:hauteurs:estjj+1}
 H(W_j)
   &\ll X_j^{1-\lambda}
   \ll H(W_{j+1})^{\theta(1-\lambda)}
   \ll X_{j+1}^{\theta(1-\lambda)^2}.
\end{align}
\end{corollary}

\begin{proof}
The first assertion follows directly from Lemma \ref{lemma:normeDi}
and the above proposition.  To prove the second one, we fix
consecutive integers $i<j$ in $I$ with $W_i+W_{i+1}\neq
W_j+W_{j+1}$, and go back to the general estimate
\eqref{est:generale} from Lemma \ref{lemma:paire}:
\begin{equation}
 \label{cor:hauteurs:eq1}
 X_i X_j
  \ll H(W_i) H(W_j) H(W_{j+1}).
\end{equation}
On the right hand side of this inequality, we apply the standard
estimate $H(W_i)\ll X_i^{1-\lambda}$ from Lemma
\ref{lemma:hauteurWi} as an upper bound for $H(W_i)$, and the
estimate $H(W_{j+1}) \ll H(W_j)^\theta$ coming from
\eqref{est:Xi+1versusXi} as an upper bound for $H(W_{j+1})$. On the
left hand side, we use instead the estimate $H(W_j)\ll
X_j^{1-\lambda}$ from Lemma \ref{lemma:hauteurWi} as a lower bound
for $X_j$. This gives
$$
 X_i^\lambda
   \ll H(W_j)^{\theta+1-1/(1-\lambda)}
   = H(W_j)^{\theta-1/\theta},
$$
and \eqref{cor:hauteurs:estij} follows.  To prove
\eqref{cor:hauteurs:estjj+1}, we note instead that, $i$ and $j$
being consecutive elements of $I$, we have $W_j=W_{i+1}$ and so
\eqref{cor:hauteurs:eq1} combined with Lemma \ref{lemma:hauteurWi}
gives
$$
 X_i X_j
  \ll H(W_i) H(W_{i+1}) H(W_{j+1})
  \ll ( X_i X_{i+1} )^{1-\lambda} H(W_{j+1}).
$$
Moving on the left all powers of $X_i$ and using the estimate
$X_{i+1}\ll X_i^\theta$ from \eqref{est:Xi+1versusXi} as a lower
bound for $X_i$, we obtain
$$
 X_{i+1}^{\lambda/\theta} X_j
  \ll X_{i+1}^{1-\lambda} H(W_{j+1}).
$$
Moving all powers of $X_{i+1}$ on the right and observing that the
exponent $1-\lambda-\lambda/\theta = 1-1/\theta$ is $\ge 0$ (since
$\theta\ge 1$), we obtain finally
$$
 X_j
  \ll X_{i+1}^{1-1/\theta} H(W_{j+1})
  \le X_j^{1-1/\theta} H(W_{j+1})
$$
which implies \eqref{cor:hauteurs:estjj+1}.
\end{proof}

%%%%%%%%%%%%%%%%%%%%%%%%%%%%%%%%%%%%%%%%%%%%%%%%%%%%%%%%%%%%%%%%%
%
%   The set $J$
%
%%%%%%%%%%%%%%%%%%%%%%%%%%%%%%%%%%%%%%%%%%%%%%%%%%%%%%%%%%%%%%%%%

\section{The set $J$}
\label{sec:J}

We assume from now on that  $\lambda > \sqrt{2}-1$. Then, for each
sufficiently large index $i$, the subspace $V_i= \langle \uxmoins_i,
\uxplus_i \rangle_\bR$ of $\bR^3$ has dimension 2 and, by Corollary
\ref{cor:hauteurs}, we have $V_i\neq V_{i+1}$. Consequently,
$\C{i,i+1}$ is a non-zero point of $\bZ^2$ for each $i\gg 1$.

\begin{notation}
Let $J$ be the set of all elements $i$ of $I$ whose successor $j$ in
$I$ satisfies $W_j+W_{j+1}\neq W_i+W_{i+1}$.
\end{notation}

By Lemma \ref{lemma:paire}, the set $J$ is infinite.  The next
result studies a possible configuration of points.

\begin{lemma}
 \label{lemma:triple}
Suppose that $\lambda > \sqrt{2}-1$, and that $h<i<j$ are three
consecutive elements of $I$ with $h\in J$ and $i\in J$. Then we have
$$
 L(\C{i,i+1})
  \ll X_{j+1}^\alpha
 \quad
 \text{where}
 \quad
 \alpha = \frac{-\lambda^4+\lambda^3+\lambda^2-3\lambda+1}%
               {\lambda(\lambda^2-\lambda+1)}.
$$
\end{lemma}

\begin{proof}
By Lemma \ref{lemma:L(Ci,i+1)}, we have
\begin{equation}
 \label{lemma:triple:inegalite1}
 L(\C{i,i+1})
  \ll X_i X_j^{-\lambda} X_{j+1}^{-\lambda}.
\end{equation}
Since $i\in J$, we have $W_i+W_{i+1}\neq W_j+W_{j+1}$, and the
second part of \eqref{est:paire} in Lemma \ref{lemma:paire} gives
\begin{equation*}
 \label{lemma:triple:inegalite2}
 X_i \ll X_j^{-1} X_{j+1}^{\theta}.
\end{equation*}
Since $h\in J$, we also have $W_h+W_{h+1}\neq W_i+W_{i+1}$, and the
estimates \eqref{cor:hauteurs:estjj+1} of Corollary
\ref{cor:hauteurs} applied to the pair $(h,i)$ instead of $(i,j)$
lead to
\begin{equation*}
 \label{lemma:triple:inegalite3}
 X_i
  \ll X_{i+1}^{(1-\lambda)\theta}
  \le X_j^{(1-\lambda)\theta}.
\end{equation*}
Put $\beta = (1-\lambda)/(\lambda^2-\lambda+1)$.  Since $\lambda \le
1/2$, we have $\beta \ge 1-\lambda \ge 1/2$.  We consider two cases.

\medskip
(a) If $X_j \ge X_{j+1}^\beta$, we substitute into
\eqref{lemma:triple:inegalite1} the first of the above two upper
bounds for $X_i$. This gives
$$
 L(\C{i,i+1})
  \ll X_j^{-1-\lambda} X_{j+1}^{\theta-\lambda}
  \le X_{j+1}^{-(1+\lambda)\beta+\theta-\lambda}
  = X_{j+1}^\alpha.
$$

\medskip
(b) If on the contrary, we have $X_j < X_{j+1}^\beta$, we substitute
instead into \eqref{lemma:triple:inegalite1} the second upper bound
for $X_i$.  Again we find
$$
 L(\C{i,i+1})
  \ll X_j^{(1-\lambda)\theta-\lambda} X_{j+1}^{-\lambda}
  \le X_{j+1}^{((1-\lambda)\theta-\lambda)\beta - \lambda}
  = X_{j+1}^\alpha,
$$
upon noting that the exponent $(1-\lambda)\theta-\lambda
=(1-2\lambda)/\lambda$ is $\ge 0$.
\end{proof}

\begin{proposition}
 \label{proposition:quadruples}
Suppose that $\lambda > \lambda_2$ where $\lambda_2 \cong 0.4241$
denotes the positive root of the polynomial $P_2(T) =
3T^4-4T^3+2T^2+2T-1$, and let $\alpha$ be as in Lemma
\ref{lemma:triple}.  Then, we have $1-2\lambda+\alpha <0$ and, for
any triple of consecutive elements $h<i<j$ of $I$ contained in $J$,
with $i$ large enough, the points $\C{i,i+1}$ and $\C{j,j+1}$ are
linearly dependent over $\bQ$.
\end{proposition}

The fact that $P_2(T)$ admits exactly one positive root $\lambda_2$
follows by observing that its second derivative $P_2''(T)=(6T-2)^2$
is non-negative on $\bR$ and that $P_2(0)$ is negative.
Consequently, if $\lambda > \lambda_2$, we have $P_2(\lambda)>0$.

\begin{proof}
For any triple of consecutive elements $h<i<j$ of $I$ contained in
$J$, Lemma \ref{lemma:triple} gives $L(\C{i,i+1})\ll X_{j+1}^\alpha$
and $L(\C{j,j+1})\ll X_{k+1}^\alpha$ where $k$ denotes the successor
of $j$ in $I$.  As the general estimates of Lemma \ref{lemma:Cij}
provide $\|\C{\ell,\ell+1}\| \ll X_{\ell+1}^{1-2\lambda}$ for each
$\ell\ge 1$, we deduce that
\begin{align*}
 |\det(\C{i,i+1},\C{j,j+1})|
   &\ll \|\C{i,i+1}\| L(\C{j,j+1}) + \|\C{j,j+1}\| L(\C{i,i+1})\\
   &\ll X_{i+1}^{1-2\lambda} X_{k+1}^\alpha
      + X_{j+1}^{1-2\lambda+\alpha} \\
   &\ll X_{k+1}^{1-2\lambda+\alpha}
      + X_{j+1}^{1-2\lambda+\alpha}.
\end{align*}
As a short computation gives $1-2\lambda+\alpha = -P_2(\lambda) /
(\lambda(\lambda^2-\lambda+1)) < 0$, we conclude that the integer
$\det(\C{i,i+1},\C{j,j+1})$ vanishes if $i$ is sufficiently large.
\end{proof}

\begin{corollary}
 \label{cor:JmoinsI}
Suppose that $\lambda>\lambda_2$.  Then the complement of $J$ in $I$
is infinite.
\end{corollary}

\begin{proof}
If $I\setminus J$ were a finite set, then, by the above proposition,
all points $\C{i,i+1}$ with $i\in I$ sufficiently large would belong
to the same one-dimensional subspace of $\bR^2$. By Lemma
\ref{lemma:fonctionL}, this would imply that $L(\C{i,i+1}) \gg 1$,
against the estimates of Lemma \ref{lemma:triple} since $\alpha <
2\lambda-1 \le 0$.
\end{proof}

%%%%%%%%%%%%%%%%%%%%%%%%%%%%%%%%%%%%%%%%%%%%%%%%%%%%%%%%%%%%%%%%%
%
%   Proof of the theorem
%
%%%%%%%%%%%%%%%%%%%%%%%%%%%%%%%%%%%%%%%%%%%%%%%%%%%%%%%%%%%%%%%%%

\section{Proof of the theorem}
\label{sec:proof}

We may assume that $\lambda > \lambda_2 \cong 0.4241 > \sqrt{2}-1$.
Then, by Corollary \ref{cor:JmoinsI}, there exist infinitely many
triples of elements $g < i < j$ of $I$ with $i$ and $j$ consecutive
satisfying
\begin{equation}
 \label{proof:suiteWW}
 W_g+W_{g+1} = W_i+W_{i+1} \neq W_j+W_{j+1}.
\end{equation}
Fix such a triple.  Since $i$ and $j$ are consecutive elements of
$I$, we have $W_{i+1}=W_j$ and so
$$
 W_j = (W_i+W_{i+1}) \cap (W_j+W_{j+1})
       = (W_g+W_{g+1}) \cap (W_j+W_{j+1}).
$$
Since the sum of $W_g+W_{g+1}$ and $W_j+W_{j+1}$ is the whole of
$\bR^4$ and that $H(\bR^4)=1$, an application of
\eqref{est:height_sum_inter} gives
\begin{equation}
 \label{proof:eq0}
 H(W_j)
  \ll H(W_g+W_{g+1}) H(W_j+W_{j+1}).
\end{equation}
By Lemma \ref{lemma:Hsomme}, we have
\begin{equation*}
% \label{proof:eq2}
 H(W_g+W_{g+1}) \ll H(W_g)^{-1/\theta} H(W_{g+1})
 \et
 H(W_j+W_{j+1}) \ll H(W_j)^{-1/\theta} H(W_{j+1}),
\end{equation*}
while the estimates \eqref{est:Xi+1versusXi} of Lemma
\ref{lemma:normeDi} provide
\begin{equation*}
% \label{proof:eq4}
 H(W_{g+1})\ll H(W_g)^\theta
 \et
 H(W_{j+1})\ll H(W_j)^\theta.
\end{equation*}
Using the latter relations respectively as a lower bound for
$H(W_g)$ and as an upper bound for $H(W_{j+1})$ and substituting
them into the former, we obtain
\begin{equation}
 \label{proof:eq6}
 H(W_g+W_{g+1}) \ll H(W_{g+1})^{1-1/\theta^2}
 \et
 H(W_j+W_{j+1}) \ll H(W_j)^{\theta-1/\theta}.
\end{equation}
Since $g<i$, we have $X_{g+1}\le X_i$ and so Lemma
\ref{lemma:hauteurWi} gives
\begin{equation}
 \label{proof:eq8}
 H(W_{g+1}) \ll c X_{g+1}^{1-\lambda} \le c X_i^{1-\lambda}.
\end{equation}
We also have
\begin{equation}
 \label{proof:eq10}
 X_i^{1-\lambda} \ll H(W_j)^{\theta^2-1}
\end{equation}
by the estimates \eqref{cor:hauteurs:estij} of Corollary
\ref{cor:hauteurs}. Combining \eqref{proof:eq0}, \eqref{proof:eq6},
\eqref{proof:eq8} and \eqref{proof:eq10}, we find
\begin{equation}
 \label{proof:eq12}
 H(W_j)
  \ll c^{1-1/\theta^2}
      H(W_j)^{(1-1/\theta^2)(\theta^2-1)+(\theta-1/\theta)}.
\end{equation}
Since \eqref{proof:eq10} shows that $H(W_j)$ tends to infinity with
$i$, we conclude that
\[
 (\theta-1/\theta)^2 + (\theta-1/\theta) \ge 1,
\]
and so $\theta-1/\theta \ge 1/\gamma$ where $\gamma=(1+\sqrt{5})/2$
(because $\theta-1/\theta$ is $\ge 0$ and we have
$1/\gamma^2+1/\gamma=1$). After simplifications, the latter relation
implies
\[
 \lambda^2-(1+2\gamma)\lambda+\gamma \ge 0.
\]
Since the polynomial $T^2-(1+2\gamma)T+\gamma$ admits two positive
real roots, $\lambda_3 \cong 0.4245$ and $\gamma/\lambda_3 \cong
3.811$, it follows that $\lambda\le \lambda_3$. Moreover, if
$\lambda=\lambda_3$, then \eqref{proof:eq12} gives $c\gg 1$, as
announced.

\begin{acknowledgment}
Part of this work was done during the workshop on Diophantine
approximation at the Lorentz Center in Summer 2003. The author
thanks the organizers for their invitation and Michel Laurent for
several discussions on the topic of the present paper.
\end{acknowledgment}

%%%%%%%%%%%%%%%%%%%%%%%%%%%%%%%%%%%%%

\end{document}